\documentclass[11pt]{article}
\usepackage{latexsym,amsfonts,amssymb,amsmath,amsthm}
\usepackage{graphicx}

\usepackage[usenames,dvipsnames]{color}
\usepackage{ulem}

\begin{document}
\setlength{\baselineskip}{16pt}

\parindent 0.5cm
\evensidemargin 0cm \oddsidemargin 0cm \topmargin 0cm \textheight 22.5cm \textwidth 16cm \footskip 2cm \headsep
0cm

\newtheorem{theorem}{Theorem}[section]
\newtheorem{lemma}{Lemma}[section]
\newtheorem{proposition}{Proposition}[section]
\newtheorem{definition}{Definition}[section]
\newtheorem{example}{Example}[section]
\newtheorem{corollary}{Corollary}[section]

\newtheorem{remark}{Remark}[section]

\numberwithin{equation}{section}

\def\p{\partial}
\def\I{\textit}
\def\R{\mathbb R}
\def\C{\mathbb C}
\def\u{\underline}
\def\l{\lambda}
\def\a{\alpha}
\def\O{\Omega}
\def\e{\epsilon}
\def\ls{\lambda^*}
\def\D{\displaystyle}
\def\wyx{ \frac{w(y,t)}{w(x,t)}}
\def\imp{\Rightarrow}
\def\tE{\tilde E}
\def\tX{\tilde X}
\def\tH{\tilde H}
\def\tu{\tilde u}
\def\d{\mathcal D}
\def\aa{\mathcal A}
\def\DH{\mathcal D(\tH)}
\def\bE{\bar E}
\def\bH{\bar H}
\def\M{\mathcal M}
\renewcommand{\labelenumi}{(\arabic{enumi})}

\def\disp{\displaystyle}
\def\undertex#1{$\underline{\hbox{#1}}$}
\def\card{\mathop{\hbox{card}}}
\def\sgn{\mathop{\hbox{sgn}}}
\def\exp{\mathop{\hbox{exp}}}
\def\OFP{(\Omega,{\cal F},\PP)}
\newcommand\JM{Mierczy\'nski}
\newcommand\RR{\ensuremath{\mathbb{R}}}
\newcommand\CC{\ensuremath{\mathbb{C}}}
\newcommand\QQ{\ensuremath{\mathbb{Q}}}
\newcommand\ZZ{\ensuremath{\mathbb{Z}}}
\newcommand\NN{\ensuremath{\mathbb{N}}}
\newcommand\PP{\ensuremath{\mathbb{P}}}
\newcommand\abs[1]{\ensuremath{\lvert#1\rvert}}
\newcommand\normf[1]{\ensuremath{\lVert#1\rVert_{f}}}
\newcommand\normfRb[1]{\ensuremath{\lVert#1\rVert_{f,R_b}}}
\newcommand\normfRbone[1]{\ensuremath{\lVert#1\rVert_{f, R_{b_1}}}}
\newcommand\normfRbtwo[1]{\ensuremath{\lVert#1\rVert_{f,R_{b_2}}}}
\newcommand\normtwo[1]{\ensuremath{\lVert#1\rVert_{2}}}
\newcommand\norminfty[1]{\ensuremath{\lVert#1\rVert_{\infty}}}

\title{Stability of
 transition waves  and   positive entire  solutions of  Fisher-KPP equations with time and space dependence}

\author{
 Wenxian Shen\\
Department of Mathematics and Statistics\\
Auburn University\\
Auburn University, AL 36849\\
U.S.A. }

\date{}
\maketitle

\noindent {\bf Abstract.} This paper is concerned with the stability of transition waves and   strictly positive entire  solutions
of random and nonlocal dispersal evolution equations of Fisher-KPP type with general
time  and space dependence, including time and space periodic or almost periodic dependence as special cases.
We first show the existence, uniqueness, and stability of  strictly positive entire solutions of such equations.
Next, we show the stability of uniformly continuous  transition waves connecting the unique  strictly positive entire solution and  the trivial solution zero
and
satisfying certain decay property at the end close to the trivial solution zero (if it exists).
The existence of transition waves has been studied in \cite{LiZh1,Nad, NoRuXi, NoXi, Wei2} for random dispersal
Fisher-KPP equations with time and space periodic dependence, in \cite{NaRo1, NaRo2, NaRo3, She6,She7,She8,TaZhZl,Zla}  for random dispersal Fisher-KPP equations with quite general time and/or  space dependence,  and in \cite{CoDaMa,RaShZh,ShZh1} for nonlocal dispersal Fisher-KPP equations with time and/or space periodic dependence.
The stability result established in this paper implies  that the transition waves obtained in many of the  above mentioned papers are asymptotically stable {for well-fitted perturbation.} Up to the author's knowledge,
it is the first time that the stability of transition waves of Fisher-KPP equations with  general time and space dependence
is studied.

\medskip


\medskip


\section{Introduction}
\setcounter{equation}{0}

The current paper is devoted to the study of the stability of transition waves and entire positive solutions of dispersal evolution equations of the form,
\begin{equation}
\label{main-eq}
 \frac{\p u}{\p t}=\mathcal{A} u+uf(t,x,u), \quad x\in \RR,
\end{equation}
where
$\mathcal{A}u(t,x)=u_{xx}(t,x)$ or $\mathcal{A}u(t,x)=\int_{\RR}\kappa(y-x)u(t,y)dy-u(t,x)$ for some nonnegative smooth function
$\kappa(\cdot)$ with $\kappa(z)>0$ for $\|z\|<r_0$ and some $r_0>0$, $\kappa(z)=0$ for $\|z\|\ge r_0$,  and $\int_{\RR}\kappa(y)dy=1$,  and  $f(t,x,u)$ is of Fisher-KPP type
in $u$.
 More precisely,  we assume that $f(t,x,u)$ satisfies  the following standing assumption.

\medskip

\noindent{\bf (H0)} {\it $f(t,x,u)$ is globally H\"older continuous in $t$ uniformly with respect to $x\in\RR$ and $u$ in bounded sets, is globally Lipschitz continuous
in $x$ uniformly with respect to $t\in\RR$ and $u$ in bounded sets, and is differentiable in $u$ with $f_u(t,x,u)$ being  bounded and uniformly continuous
in  $t\in\RR$, $x\in\RR$, and $u$ in bounded sets.
There are $\beta_0>0$ and
$P_0>0$ such that $f(t,x,u)\leq -\beta_0$ for $t,x\in\RR$ and $u\geq P_0 $ and
$\frac{\p f}{\p u} (t,x,u)\leq -\beta_0$ for $t,x\in\RR$ and $u\geq 0$. Moreover,
\begin{equation}
\label{condition-eq1}
-\infty<\inf_{t\in\RR,x\in\RR,0\le u\le M} f(t,x,u)\le \sup_{t\in\RR,x\in\RR,0\le u\le M}f(t,x,u)<\infty
\end{equation}
for all $M>0$, and
\begin{equation}
\label{condition-eq2} \liminf_{t-s\to\infty}\frac{1}{t-s}\int_s^t
\inf_{x\in\RR} f(\tau,x,0)d\tau>0.
\end{equation}
}

\vspace{-0.1in} Equations \eqref{main-eq} appears in the study of population dynamics of species in biology (see \cite{ArWe1}, \cite{ArWe2}, \cite{BeHaRo}, \cite{Fisher},
\cite{KPP}), where  $u(t,x)$ represents  the population density of a species
at time $t$ and space location $x$, $\mathcal{A}u$ describes the dispersal or movement of the organisms  and $f(t,x,u)$ describes the growth rate
of the population. When $\mathcal{A}u(t,x)=u_{xx}$, it indicates that the movement of the organisms occurs between adjacent locations randomly and the dispersal in this case is referred to as
{\it random dispersal}.
When $\mathcal{A}u(t,x)=\int_{\RR}\kappa(y-x)u(t,y)dy-u(t,x)$, it indicates that the movement of the organisms occurs between adjacent as well as non-adjacent locations
and the dispersal in this case  is referred to as {\it nonlocal dispersal}.
The time and space dependence of the equation reflects the heterogeneity of the underlying environments.

Because of biological reason,  only  nonnegative solutions of \eqref{main-eq} will be considered throughout this paper.
Also, by a solution of \eqref{main-eq} in this paper, we always mean a classical solution, i.e., a solution satisfies \eqref{main-eq} in the classical sense,
unless otherwise specified.
A function $u(t,x)$ is called an {\it entire solution} if it is a bounded and continuous function on $\RR\times\RR$ and satisfies \eqref{main-eq}
for $(t,x)\in\RR\times\RR$. An entire solution $u(t,x)$ is called {\it strictly positive} if $\inf_{(t,x)\in\RR\times\RR} u(t,x)>0$.
In the following, we may call a strictly positive entire solution  a {\it positive entire solution} if no confusion occurs.

Equation \eqref{main-eq} satisfying (H0) is called in literature a Fisher-KPP type  equation
 due to the pioneering papers of Fisher
\cite{Fisher} and Kolmogorov, Petrowsky, Piscunov \cite{KPP} on the following special case of \eqref{main-eq},
\begin{equation}
\label{classical-fisher-eq}
 \frac{\p u}{\p t}=\frac{\p ^ 2u}{\p x^2}+u(1-u),\quad\quad x\in \RR.
\end{equation}
It is clear that $u(t,x)\equiv 1$ is a unique strictly positive entire solution of \eqref{classical-fisher-eq}. Moreover, it is  globally stable with respect to
strictly positive initial data $u_0(x)$.

The existence, uniqueness, and stability of positive entire solutions of \eqref{main-eq} is one of the central problems about the dynamics of \eqref{main-eq}. Assume (H0). We show in this paper that

\medskip

\noindent $\bullet$ {\it  \eqref{main-eq} has a unique stable strictly  positive entire solution $u^+(t,x)$. If, in addition,
$f(t,x,u)$ is periodic in $t$ and/or $x$, then so is $u^+(t,x)$, and if $f(t,x,u)$ and $f_u(t,x,u)$ are almost periodic in $t$ and/or  $x$, then so if $u^+(t,x)$ }(see Theorem \ref{positive-solu-thm}).

 \medskip

 It should be pointed out that when  the dispersal is random, there are many studies on the positive entire solutions of \eqref{main-eq}
 for various special cases (see \cite{BeHaRoq, BeHaRos, CaCo, MiSh2, MiSh3, She8}, etc.). When the dispersal is nonlocal, there are also several studies on the positive entire solutions of \eqref{main-eq} for some special cases (see \cite{BaZh, BeCoVo, KoSh, RaSh, ShZh1, ShZh2},  etc.)
  It should also be pointed out that, in the random dispersal case, by the regularity and a priori estimates for parabolic equations
  (see \cite{Fri}), it is easy to prove the continuity of bounded solutions. In the nonlocal dispersal case, due to the lack of regularity
of solutions, the proof of the  continuity of bounded solutions is not trivial.
Thanks to the existence, uniqueness, and stability of a unique strictly
positive entire solution, \eqref{main-eq} is also said to be of {\it monostable type}.

The traveling wave problem is also among the central problems about the dynamics of \eqref{main-eq}.
This problem is  well understood for the classical Fisher or KPP equation \eqref{classical-fisher-eq}.
For example,
Fisher in
\cite{Fisher} found traveling wave solutions $u(t,x)=\phi(x-ct)$,
$(\phi(-\infty)=1,\phi(\infty)=0)$ of all speeds $c\geq 2$ and
showed that there are no such traveling wave solutions of slower
speed. He conjectured that the take-over occurs at the asymptotic
speed $2$. This conjecture was proved in \cite{KPP}  {for some special initial distribution and was proved in \cite{ArWe2} for the general case.
 More precisely, it is proved
in \cite{KPP} that for the  nonnegative solution $u(t,x)$ of \eqref{classical-fisher-eq} with
$u(0,x)=1$ for $x<0$ and $u(0,x)=0$ for $x>0$, $\lim_{t\to \infty}u(t,ct)$ is $0$ if $c>2$ and $1$ if $c<2$. It is proved
in \cite{ArWe2} that for any
nonnegative solution $u(t,x)$ of (\ref{classical-fisher-eq}), if at
time $t=0$, $u$ is $1$ near $-\infty$ and $0$ near $\infty$, then
$\lim_{t\to \infty}u(t,ct)$ is $0$ if $c>2$ and $1$ if $c<2$.} Put $c^*=2$. $c^*$ is of the following spatially
spreading property: for any nonnegative solution $u(t,x)$ of
\eqref{classical-fisher-eq}, if at time $t=0$, $u(0,x)\geq \sigma$
for some $\sigma>0$ and $x\ll -1$  and $u(0,x)=0$ for $x\gg 1$, then
$$
\inf_{x\leq  c^{'}t}|u(t,x)- 1|\to 0\quad \forall c^{'}<c^*\quad {\rm and}\quad \sup_{x\geq
c^{''}t}u(t,x)\to 0 \quad\forall c^{''}>c^* \quad {\rm as}\quad t\to\infty.
$$
In
literature, $c^*$ is hence  called the {\it
spreading speed} for (\ref{classical-fisher-eq}).
The results on traveling wave solutions of \eqref{classical-fisher-eq} have been well extended to general time and space independent
monostable equations (see \cite{ArWe1}, \cite{ArWe2}, \cite{Bra}, \cite{Ham},  \cite{Kam}, \cite{Sat}, \cite{Uch}, etc.).

Due to the inhomogeneity of the underlying media of  biological
models in nature,  the investigation of
the traveling wave problem for time and/or space dependent
dispersal evolution equations is gaining  more and more attention. The notion of transition waves or generalized traveling waves
 has been introduced for dispersal evolution equations with  general time and space
 dependence (see Definition \ref{transition-wave-def} and Remark \ref{transition-wave-rk}), which naturally extends the notion of traveling wave solutions for time and space independent dispersal evolution
equations  to  the equations with  general time and space
 dependence.
 A huge amount of research has been
carried out toward the transition waves or generalized traveling waves of
various time and/or space dependent monostable equations. See, for example,
 \cite{BeHaNa1, BeHaNa2, BeHaRo, BeHa07, BeHa12, FrGa, Ham, HePaSt, HuSh,  HuZi1,
LiYiZh, LiZh, LiZh1, Mat, Nad, NaRo1, NaRo2, NaRo3, NoRoRyZl, NoRuXi, NoXi, She6, She7, She8, TaZhZl, Wei1, Wei2, Xin1, Zla},  and references therein for
space and/or time dependent Fisher-KPP
type equations with random dispersal, and see, for example, \cite{CaCh, CoDu, CoDaMa, LiZl, RaShZh, ShSh1, ShSh2, ShZh1, ShZh2}, and references therein
for space and/or time dependent Fisher-KPP type equations with nonlocal dispersal.

 {It should be pointed out that the works \cite{HaRo}, \cite{She7}, \cite{ShSh1}, and \cite{ShZh1} considered the stability of transition waves in spatially periodic and time independent or spatially homogeneous and time
  dependent Fisher-KPP type equations with random and nonlocal dispersal. In particular,   the stability of transition  waves in spatially periodic and time independent Fisher-KPP type equations with random dispersal (resp. nonlocal dispersal) is studied in \cite{HaRo} (resp. \cite{ShZh1}) and the stability of transition waves in  spatially homogeneous and time almost periodic  Fisher-KPP type equations with random dispersal
  (resp. nonlocal dispersal) is investigated in \cite{She7} (resp. \cite{ShSh1}).  The paper \cite{HaRo} also considered the stability of traveling waves in spatially and temporally periodic Fisher-KPP equations with random dispersal (see \cite[Section 1.4]{HaRo}).}

However, as long as the equations  depend on both time and space variables non-periodically,  all the existing
works are on the existence of transition waves or generalized traveling waves and there is little on the stability of transition
waves.  In the current paper, we consider
 the stability of transition waves of Fisher-KPP equations with general time and space dependence
(see Definition \ref{transition-wave-def} for the definition of transition waves). We show that

\medskip

\noindent $\bullet$
{\it Any transition wave of \eqref{main-eq} connecting $u^+(t,x)$ and $0$ and  satisfying certain decaying property near $0$ is asymptotically stable { for well-fitted perturbation}}
(see Theorem \ref{main-thm-general-case} for detail).

\medskip

{We point out that the existence of transition waves of \eqref{main-eq} with non-periodic time and/or space dependence  has been studied in \cite{LiZl, NaRo1, NaRo2, NaRo3, She7, She8, ShSh1, Zla}. Applying the above stability result for general transition waves of \eqref{main-eq}, we prove}

\medskip

\noindent $\bullet$ {\it The non-critical transition waves  established in \cite{LiZl, NaRo1, NaRo2, NaRo3, She7, She8, ShSh1, Zla}
are asymptotically stable {for well-fitted perturbation}}
(see Theorem \ref{main-thm-special-case} and Remark \ref{special-case-rk1} for detail).

\medskip

Up to the author's knowledge,  it is the first time that the stability of transition waves of Fisher-KPP type equations with general time and space dependence
is studied.
Among the technical tools used in the proofs of the main results are spectral theory for
linear  dispersal evolution equations with time and space dependence, comparison principle, and the very nontrivial application of the so called part metric.

The rest of the paper is organized as follows. In section 2, we will introduce the standing notations, definitions,  and state the main results
of the paper. We study the existence, uniqueness and stability of positive entire solutions in section 3. Sections 4 and 5 are devoted to the proofs of the main results
 on transition waves.

\medskip

\section{Notations, definitions,  and main results}

In this section, we  introduce the standing notations, definitions,  and state the main results
of the paper. Throughout this section, we assume that (H0) holds.

First of all,  we  recall the definition of almost periodic
functions.

\begin{definition}[Almost periodic function]
\label{almost-periodic-def}
\begin{itemize}

\item[(1)] A continuous function $g:\RR\to \RR$ is called {\rm almost periodic} if
for any $\epsilon>0$, the set
$$
T(\epsilon)=\{\tau\in\RR\,|\, |g(t+\tau)-g(t)|<\epsilon\,\, \, \text{for all}\,\, t\in\RR\}
$$
is relatively dense in $\RR$.

\item[(2)] Let $g(t,x,u)$ be a continuous function of $(t,x,u)\in\RR\times\RR^m\times\RR^n$. $g$ is said to be {\rm almost periodic in $t$ uniformly with respect to $x\in\RR^m$ and
$u$ in bounded sets} if
$g$ is uniformly continuous in $t\in\RR$,  $x\in\RR^m$, and $u$ in bounded sets and for each $x\in\RR^m$ and $u\in\RR^n$, $g(t,x,u)$ is almost periodic in $t$.

\item[(3)] For a given almost periodic function $g(t,x,u)$, the hull $H(g)$ of $g$ is defined by
\begin{align*}
H(g)=\{\tilde g(\cdot,\cdot,\cdot)\,|\, & \exists t_n\to\infty \,\,\text{such that}\,\, g(t+t_n,x,u)\to \tilde g(t,x,u)\,\, \text{uniformly in}\,\, t\in\RR\,\, {\rm and}\\
&
\,\, (x,u)\,\, \text{in bounded sets}\}.
\end{align*}
\end{itemize}
\end{definition}

\begin{remark}
\label{almost-periodic-rk}
Let $g(t,x,u)$ be a continuous function of $(t,x,u)\in\RR\times\RR^m\times\RR^n$. $g$ is almost periodic in $t$ uniformly with respect to
$x\in\RR^m$ and $u$ in bounded sets  if and only if
 $g$ is uniformly continuous in $t\in\RR$,  $x\in\RR^m$, and $u$ in bounded sets and for any sequences $\{\alpha_n^{'}\}$,
$\{\beta_n^{'}\}\subset \RR$, there are subsequences $\{\alpha_n\}\subset\{\alpha_n^{'}\}$, $\{\beta_n\}\subset\{\beta_n^{'}\}$
such that
$$
\lim_{n\to\infty}\lim_{m\to\infty}g(t+\alpha_n+\beta_m,x,u)=\lim_{n\to\infty}g(t+\alpha_n+\beta_n,x,u)
$$
for each $(t,x,u)\in\RR\times\RR^m\times\RR^n$ (see \cite[Theorems 1.17 and 2.10]{Fin}).
\end{remark}

Next, let
$$C_{\rm uinf}^b(\RR)=\{u\in C(\RR,\RR)\,|\, u(x)\,\,\,\text{is uniformly continuous and bounded on}\,\, \RR\}
$$
endowed with the norm $\|u\|_\infty=\sup_{x\in\RR}|u(x)|$.
By general semigroup theory (see \cite{Hen}), for any $u_0\in C_{\rm unif}^b(\RR)$, \eqref{main-eq} has a unique (local) solution
 $u(t,x;t_0,u_0)$ with $u(t_0,x;t_0,u_0)=u_0(x)$ for $x\in\RR$.

 We then consider the existence, uniqueness, and stability of strictly positive entire solutions of \eqref{main-eq}.
The following is the main result of the paper on the existence, uniqueness, and stability of strictly positive entire solution of \eqref{main-eq}.

\begin{theorem}
\label{positive-solu-thm} There is a unique bounded  strictly
positive entire solution $u^+(t,x)$ of \eqref{main-eq}  with
$u^+(t,x)$ being uniformly continuous in $(t,x)\in\RR\times \RR$.
Moreover, for any given $u_0\in C_{\rm unif}^b(\RR)$ with
$\inf_{x\in\RR}u_0(x)>0$,
$$
\lim_{t\to\infty}\|u(t+t_0,\cdot;t_0,u_0)-u^+(t+t_0,\cdot)\|_\infty=0
$$
 uniformly in $t_0\in\RR$. If, in addition, $f(t,x,u)$ is periodic
in $t$ (resp. periodic in $x$), then so is $u^+(t,x)$. If $f(t,x,u)$ and $f_u(t,x,u)$
are almost periodic in $t$ (resp.\ almost periodic in $x$), then so
is $u^+(t,x)$.
\end{theorem}

We now consider transition waves of \eqref{main-eq} connecting $u^+(t,x)$ and $0$.

\begin{definition}
\label{transition-wave-def}
\begin{itemize}
\item[(1)]
 An entire solution $u=U(t,x)$ of \eqref{main-eq} is called a  {\rm transition wave} (connecting $0$ and $u^+(\cdot,\cdot)$)
  if $U(t,x)\in(0,u^+(t,x))$ for all $(t,x)\in\RR\times\RR$,  and there exists a function $X:\RR\to\RR$, called {\rm interface location function}, such that
\begin{equation*}
\lim_{x\to-\infty}U(t,x+X(t))=u^+(t,x+X(t))\,\,\text{and}\,\,\lim_{x\to\infty}U(t,x+X(t))=0\,\,\text{uniformly in}\,\,t\in\R.
\end{equation*}

\item[(2)] {Assume that $u(t,x)=U(t,x)$ is a transition wave of \eqref{main-eq} with $X:\RR\to\RR$ being an interface location function. $c:=\liminf_{t-s\to\infty,t>s}\frac{X(t)-X(s)}{t-s}$ is called
the {\rm least mean speed} of the transition wave. If $\lim_{t-s\to\infty, t>s}\frac{X(t)-X(s)}{t-s}$ exists, $c:= \lim_{t-s\to\infty, t>s}\frac{X(t)-X(s)}{t-s}$ is called {\rm the average speed or mean speed} of the transition wave.}

\item[(3)] {Assume that $u(t,x)=U(t,x)$ is a transition wave of \eqref{main-eq}. It is called {\rm asymptotically stable} if for any $t_0\in\RR$ and $u_0\in C_{\rm unif}^b(\RR)$ satisfying that
$u_0(x)>0$ for all $x\in\RR$ and
\begin{equation}
\label{cond-eq3}
\inf_{x\le x_0}u_0(x)>0\quad \forall\,\, x_0\in\RR,\quad \lim_{x\to\infty}\frac{u_0(x)}{U(t_0,x)}=1,
\end{equation}
there holds
\begin{equation}
\label{cond-eq4}
\lim_{t\to\infty}\|\frac{u(t+t_0,\cdot;t_0,u_0)}{U(t+t_0,\cdot)}-1\|_{C_{\rm unif}^b(\RR)}=0.
\end{equation}}
\end{itemize}
\end{definition}

\begin{remark}
\label{transition-wave-rk}
\begin{itemize}
\item[(1)]  The interface location function $X(t)$ of a transition wave $u=U(t,x)$  tells the position of the transition front of $U(t,x)$ as time $t$ elapses, while the uniform-in-$t$ limits (the essential property in the definition) shows the \textit{bounded interface width}, that is,
\begin{equation}
\label{interface-width-eq}
\forall\,\,0<\epsilon_{1}\leq\epsilon_{2}<1,\quad\sup_{t\in\RR}{\rm diam}\{x\in\RR|\epsilon_{1}\leq U(t,x)\leq\epsilon_{2}\}<\infty.
\end{equation}
Notice, if $\xi(t)$ is a bounded function, then $X(t)+\xi(t)$ is also an interface location function. Thus, interface location function is not unique. But, it is easy to check that if $Y(t)$ is another interface location function, then $X(t)-Y(t)$ is a bounded function. Hence, interface location functions are unique up to addition by bounded functions {and the least mean speed of a transition wave is well defined}.

\item[(2)] When $f(t+T,x,u)=f(t,x+p,u)=f(t,x,u)$, an entire solution $u=U(t,x)$ of \eqref{main-eq} is called a {\rm periodic traveling wave solution}
with speed $c$ and connecting $u^+(t,x)$ and $0$
if there is $\Phi(x,t,y)$ such that
$$
U(t,x)=\Phi(x-ct,t,ct),
$$
$$
\Phi(x,t+T,y)=\Phi(x,t,y+p)=\Phi(x,t,y),
$$
and
$$
\lim_{x\to -\infty} \Big(\Phi(x,t,y)-u^+(t,x+y)\Big)=0,\quad \lim_{x\to\infty} \Phi(x,t,y)=0
$$
uniformly in $t\in\RR$ and $y\in\RR$. It is clear that if $u=U(t,x)$ is a periodic traveling wave solution, then it is a transition wave.

\item[(3)] When $f(t,x,u)$ is almost periodic in $t$ and periodic in $x$ with period $p$, an entire solution $u=U(t,x)$ of \eqref{main-eq} is
called an {\rm almost periodic traveling wave solution} with average speed $c$ and connecting $u^+(t,x)$ and $0$ if there are $\xi(t)$ and $\Phi(x,t,y)$ such that
$$
U(t,x)=\Phi(x-\xi(t),t,\xi(t)),
$$
$$
\Phi(x,t,y)\quad \text{is almost periodic in}\,\, t\,\, \text{and periodic in}\,\, y,
$$
$$
\lim_{x\to -\infty} \Big(\Phi(x,t,y)-u^+(t,x+y)\Big)=0,\quad \lim_{x\to\infty} \Phi(x,t,y)=0
$$
uniformly in $t\in\RR$ and $y\in\RR$, and
$$
\lim_{t-s\to\infty}\frac{\xi(t)-\xi(s)}{t-s}=c.
$$
It is clear that if $u=U(t,x)$ is an almost  periodic traveling wave solution, then it is a transition wave.

\item[(4)] The reader is referred to \cite{BeHa07, BeHa12} for the introduction of the notion of transition waves in the general case, and
 to \cite{Mat, She4, She7, She8} for the time almost periodic or space almost periodic cases.

 \item[(5)] In the case that $\mathcal{A}u=u_{xx}$, by the regularity and a priori estimates for parabolic equations, any continuous transition wave $u=U(t,x)$ of \eqref{main-eq} is uniformly continuous in $(t,x)\in\RR\times\RR$. In the case that
     $\mathcal{A}u(t,x)=\int_{\RR}\kappa(y-x)u(t,y)dy-u(t,x)$, it is proved in \cite{ShSh2} that a transition wave $u=U(t,x)$ of \eqref{main-eq}
     is uniformly continuous under quite general conditions.
\end{itemize}
\end{remark}

We have the following general theorem on the stability of transition waves of \eqref{main-eq}.

\begin{theorem}
\label{main-thm-general-case}
Assume that $u=U(t,x)$ is a transition wave of \eqref{main-eq} with interface location $X(t)$  satisfying the following properties: $U(t,x)$ is uniformly continuous in $(t,x)\in\RR\times\RR$,
\begin{equation}
\label{cond-eq-2}
\forall\,\,\tau>0,\quad \sup_{t,s\in\RR,|t-s|\le\tau}|X(t)-X(s)|<\infty,
\end{equation}
and there are positive continuous functions $\phi(t,x)$ and $\phi_1(t,x)$ such that
\begin{equation}
\label{cond-eq-1}
\liminf_{x\to-\infty}\phi(t,x)=\infty,\quad \liminf_{x\to -\infty}\phi_1(t,x)=\infty, \quad \lim_{x\to\infty}\phi(t,x)=0,\quad \lim_{x\to\infty}\phi_1(t,x)=0,
\end{equation}
\begin{equation}
\label{cond-eq0}
\lim_{x\to -\infty}\frac{\phi(t,x+X(t))}{\phi_1(t,x+X(t))}=0,\quad \lim_{x\to\infty}\frac{\phi_1(t,x+X(t))}{\phi(t,x+X(t))}=0
\end{equation}
exponentially, and the second limit in \eqref{cond-eq0} is  uniformly in $t$;
\begin{equation}
\label{cond-eq1}
d^*\phi(t,x)-d_1^* \phi_1(t,x)\le U(t,x)
\le d^*\phi(t,x)+d_1^*\phi_1(t,x)
\end{equation}
for some $d^*,d_1^*>0$ and all $t,x\in\RR$; and
 for any given $t_0\in\RR$ and $u_0\in C_{\rm unif}^b(\RR)$ with $u_0(x)\ge 0$, if
$$
 u_0(x)\ge d\phi(t_0,x)-d_1\phi_1(t_0,x)\quad \Big({\rm resp., }\,\,
u_0(x)\le d \phi(t_0,x)+d_1
\phi_1(t_0,x)\Big)
$$
for some $0<d<2d^*$, $d_1\gg 1$, and all $x\in\RR$, then
\begin{equation}
\label{cond-eq2}
 u(t,x;t_0,u_0)\ge d\phi(t,x)-d_1 \phi_1(t,x)\quad \Big({\rm resp., }\,\, u(t,x;t_0,u_0) \le  d
\phi(t,x)+d_1 \phi_1(t,x)\Big)
\end{equation}
for all $t\ge t_0$ and $x\in\RR$.
Then the transition wave $u=U(t,x)$ is asymptotically stable.
\end{theorem}

It should be pointed out that, when $\mathcal{A}u=u_{xx}$,
by \cite[Proposition 4.2]{HaRos}, \eqref{cond-eq-2} holds for any transition wave of \eqref{main-eq}.
In the above theorem, the existence of transition waves is assumed. {The existence of transition waves of \eqref{main-eq} with
non-periodic time and/or space dependence has been studied in \cite{LiZl, NaRo1, NaRo2, NaRo3, She7, She8, ShSh1, Zla}. Applying
Theorem \ref{main-thm-general-case} or the arguments in the proof of Theorem \ref{main-thm-general-case}, we can establish the asymptotic stability of the transition waves
 proved in \cite{LiZl, NaRo1, NaRo2, NaRo3, RaShZh, She7, She8,  ShSh1, Zla}. For convenience, we introduce the following assumptions.}

\medskip

\noindent {\bf (H1)} {$f(t,x+p,u)=f(t,x,u)$ for some $p>0$, $f(t,x,1)=0$, $f(t,x,u)\le f(t,x,0)$ for $0\le u\le 1$, $\inf_{(t,x)\in\RR\times\RR} f(t,x,u)>0$ for $u\in (0,1)$, and
$f(t,x,u)\ge f(t,x,0)u-C u^{1+\nu}$ for some $C>0$, $\delta,\nu\in (0,1]$, $u\in (0,\delta)$.}

\medskip

\noindent {\bf (H2)} {$f(t,x,u)=a(x)(1-u)$, $\inf_{x\in\RR} a(x)>0$,  $a(x)$ is almost periodic in $x$, and there
exists an almost periodic positive function $\phi\in C^2(\RR)$  such that $\phi_{xx}+a(x)\phi(x)=\lambda_0\phi(x)$ for $x\in\RR$, where
$$
\lambda_0=\inf\{\lambda\in\RR\,|\, \exists \phi\in C^2(\RR),\,\, \phi>0,\,\, \phi_{xx}+a(x)\phi(x)\le\lambda \phi(x)\,\, {\rm for}\,\, x\in\RR\}.
$$}

\vspace{-0.1in} \noindent {\bf (H3)} {$f(t,x,u)=f(x,u)$, $f(x,1)=0$, $f_u(x,u)<0$,
$$
a(x)g(u)\le uf(x,u)\le a(x) u,\quad u\in [0,1],
$$
$$
g\in C^1([0,1]),\,\, g(0)=g(1)=0,\,\, g^{'}(0)=1,\,\,\, 0\le g(u)\le u\,\, {\rm for}\,\, u\in (0,1),
$$
$$
\int_0^1 \frac{u-g(u)}{u^2}du<\infty,\quad \Big(\frac{g(u)}{u}\Big)^{'}<0\,\, {\rm for}\,\, u\in (0,1),
$$
and
$$
0< a_-:=\inf a(x)\le \sup a(x):=a_+<\infty.
$$ }

{Observe that, assuming one of (H1), (H2),  and (H3),  $u^+(t,x)\equiv 1$. We have the following theorem on  the asymptotic stability of the transition waves
 established in \cite{NaRo1, NaRo2, NaRo3, RaShZh, She7, She8,  Zla}.}

\begin{theorem}
\label{main-thm-special-case}
\begin{itemize}

\item[(1)] {Assume that  $\mathcal{A}u=u_{xx}$ and that $f(t,x,u)$ satisfies (H1). Then there is $c^*\in\RR$ such that for any $c>c^*$,
  \eqref{main-eq}  has an asymptotically stable transition wave with least mean speed $c$.}

 \item[(2)] {Assume that  $\mathcal{A}u=u_{xx}$ and that $f(t,x,u)$ satisfies (H2). Then there is $c^*\in\RR$ such that for any $c>c^*$,
  \eqref{main-eq}  has an asymptotically stable transition wave with  mean speed $c$.}

 \item[(3)] {Assume that $\mathcal{A}u=u_{xx}$ and that $f(t,x,u)$ satisfies (H3). Let $\lambda_0=\sup\sigma[\p_{xx}^2+a(\cdot)]$ and $\lambda_1=2a_-$. If $\lambda_0<\lambda_1$, then for any $\lambda\in (\lambda_0,\lambda_1)$, there is an asymptotically stable  transition wave solution $U_\lambda(t,x)$ of \eqref{main-eq} satisfying that
     $$
 U_\lambda(t,x)\le v_\lambda(t,x),
     $$
     where $v_\lambda(t,x)=\phi_\lambda(x) e^{\lambda t}$ and $\phi_\lambda(x)$ is the unique solution of
$$
\phi_\lambda^{''}(x)+a(x)\phi_\lambda(x)=\lambda\phi_\lambda(x)\quad {\rm for}\quad x\in\RR,\quad
\phi_\lambda(0)=1,\quad \lim_{x\to\infty}\phi_\lambda(x)=0.
$$}


\vspace{-0.2in}\item[(4)] Assume that  $\mathcal{A}u=\int_{\RR}\kappa(y-x)u(t,y)dy-u(t,x)$, and that $f(t,x,u)$ is periodic in both $t$ and $x$, then there is $c^*>0$ such that for any $c>c^*$, there is an asymptotically stable periodic traveling wave of \eqref{main-eq} with  speed $c$.
\end{itemize}
\end{theorem}

{\begin{remark}
\label{special-case-rk1}
\begin{itemize}

\item[(1)]   The existence of transition waves in Theorem \ref{main-thm-special-case}(1) is proved in \cite{NaRo2}.  In fact,  the authors of \cite{NaRo2}  studied a more general case. With the same techniques for the proof of Theorem \ref{main-thm-general-case} the asymptotic
stability of transition waves in this general case can also be established. The following  special cases should be mentioned.
The existence of transition waves in the case that $f(t,x,u)\equiv f(t,u)$ is proved in \cite{NaRo1}.
 When $f(t,x,u)$ is  almost periodic in $t$ or recurrent and unique ergodic in $t$, the existence of transition waves
  is proved in \cite{She8}.  In \cite{She7}, stability, uniqueness, and almost periodicity of transition waves  are also proved when $f(t,x,u)\equiv f(t,u)$ is almost periodic in $t$.

\item[(2)]  The existence of transition waves in Theorem \ref{main-thm-special-case}(2) is proved in \cite{NaRo3}. Again,
 the authors of \cite{NaRo3} actually studied a more general case, and with the same techniques for the proof of Theorem \ref{main-thm-general-case} the asymptotic
stability of transition waves in this general case can also be established.

\item[(3)] The existence of transition waves in Theorem \ref{main-thm-special-case}(3) is proved in \cite{Zla}.
By the similar arguments as those in Theorem \ref{main-thm-special-case}(3), it can be proved that the transition waves established in
\cite{LiZl} for the case that $\mathcal{A}u=\int_{\RR}\kappa(y-x)u(t,y)dy-u(t,x)$ and $f(t,x,u)\equiv f(x,u)$ are asymptotically stable,
and that the transition waves established in \cite{ShSh1} for the case that $\mathcal{A}u=\int_{\RR}\kappa(y-x)u(t,y)dy-u(t,x)$ and $f(t,x,u)\equiv f(t,u)$ are asymptotically stable (the stability of transition waves in this case has been proved in \cite{ShSh1} by the ``squeezed'' techniques).

\item[(4)] The existence of transition waves in Theorem \ref{main-thm-special-case}(4) is proved in \cite{RaShZh}.  In the case
that $f(t,x,u)\equiv f(x,u)$ is periodic in $x$, the stability of transition waves is proved in \cite{ShZh1} by the ``squeezed'' techniques.

\item[(5)]  Transition waves with least mean speed $c^*$ are related to the so called {\rm critical traveling waves} in literature (see \cite{Nad1}, \cite{She4}).
Both the existence and  stability of critical transition waves are much more difficult to study. The reader is referred to \cite{HaRo} and references therein for
the study of the stability of critical traveling waves in the space and/or time periodic or time independent case with random dispersal. The reader is referred to \cite{Nad, NaRo1, NaRo2, NaRo3} for some results on the existence of critical transition waves of Fisher-KPP equations with random dispersal, and to \cite{CoDaMa} and reference therein for the  existence of
 critical transition waves in time independent and space periodic or independent  Fisher-KPP equations with nonlocal dispersal.
The stability of critical transition waves in the general time and space dependent Fisher-KPP equations with random dispersal and
in time and/or space periodic Fisher-KPP equations with nonlocal dispersal remains open.
\end{itemize}
\end{remark}
}

\section{Positive entire solutions}

In this section, we study the existence, uniqueness, and stability of positive entire  solutions and prove Theorem \ref{positive-solu-thm}.

For a given continuous and bounded function $u:[t_1,t_2)\times\RR\to\RR$, it is called a {\it super-solution} ({\it sub-solution}) of \eqref{main-eq} on $[t_1,t_2)$ if
\begin{equation}
\label{sub-super-solution-eq}
u_t(t,x)\geq (\leq) (\mathcal{A} u)(t,x)+u(t,x)f(t, x,u(t,x))\quad\forall  (t,x)\in (t_1,t_2)\times \RR.
\end{equation}

\begin{proposition}[Comparison principle]
\label{comparison-prop}
\begin{itemize}
\item[(1)]  Suppose that $u^1(t,x)$ and $u^2(t,x)$ are sub- and super-solutions of \eqref{main-eq}
on $[t_1,t_2)$ with  $u^1(t_1,x)\leq u^2(t_2,x)$ for $x\in \RR$. Then
$u^1(t,x)\leq u^2(t,x)$ for  $t\in (t_1,t_2)$ and $x\in\RR$. Moreover, if $u^1(t_1,x)\not\equiv u^2(t_2,x)$ for $x\in\RR$, then
$u^1(t,x)<u^2(t,x)$ for $t\in(t_1,t_2)$ and $x\in\RR$.

\item[(2)] If $u_{01},u_{02}\in C_{\rm unif}^b(\RR)$ and $u_{01}\leq u_{02}$, then $u(t,\cdot;t_0,u_{01})\leq u(t,\cdot;t_0,u_{02})$ for $t>t_0$ at which both
$u(t,\cdot;t_0,u_{01})$ and $u(t,\cdot;t_0,u_{02})$  exist. Moreover, if  $u_{01}\not =u_{02}$, then $u(t,x;t_0,u_{01})< u(t,x;t_0,u_{02})$ for all $x\in\RR$
 and $t>t_0$ at which both
$u(t,\cdot;t_0,u_{01})$ and $u(t,\cdot;t_0,u_{02})$  exist.

\item[(3)] If $u_{01},u_{02}\in C_{\rm unif}^b(\RR)$ and $u_{01}\ll u_{02}$ (i.e. $\inf_{x\in\RR}\big(u_{02}(x)-u_{01}(x)\big)>0$), then  $u(t,\cdot;t_0,u_{01})\ll u(t,\cdot;t_0,u_{02})$ for $t>t_0$ at which both
$u(t,\cdot;t_0,u_{01})$ and $u(t,\cdot;t_0,u_{02})$  exist.
\end{itemize}
\end{proposition}

\begin{proof}
When $\mathcal{A}u=u_{xx}$, the proposition follows from comparison principle for parabolic equations (see \cite{Fri}).
When $\mathcal{A}u(t,x)=\int_{\RR}\kappa(y-x)u(t,y)dy-u(t,x)$, it follows from comparison principle for
nonlocal dispersal evolution equations (see, for example, \cite[Propositions 2.1,2.2]{ShZh2}).
\end{proof}

Note that, by (H0), for $M\gg 1$, $u(t,x)\equiv M$ is a super-solution of \eqref{main-eq} on $\RR$.
The following proposition follows directly from Proposition \ref{comparison-prop} and (H0).

\begin{proposition}
\label{global-existence-prop} For any $t_0\in\RR$ and $u_0\in C_{\rm
unif}^b(\RR)$ with $u_0(\cdot)\ge 0$, $u(t,x;t_0,u_0)$ exists for
all $t\ge t_0$ and $\sup_{t\ge
0}\|u(t+t_0,\cdot;t_0,u_0)\|_\infty<\infty$.
\end{proposition}

For given $u,v\in C_{\rm unif}^b(\RR)$ with $u,v\ge 0$, if there is $\alpha_0\ge 1$ such that
$$
\frac{1}{\alpha_0} v(x)\le u(x)\le \alpha_0
v(x)\quad\forall\,\,x\in\RR,
$$
then we can define the so called {\it part metric} $\rho(u,v)$ between $u$ and $v$ by
$$
\rho(u,v)=\inf\{\ln \alpha\,|\, \alpha\ge 1,\,\, \frac{1}{\alpha}v(\cdot)\le u(\cdot)\le \alpha v(\cdot)\}.
$$
Note that for any given $u,v\in C_{\rm unif}^b(\RR)$ with $u,v\ge 0$, the part metric between $u$ and $v$ may not be defined.

\begin{proposition}
\label{part-metric-prop}
\begin{itemize}
\item[(1)]
For given $u_0,v_0\in C_{\rm unif}^b(\RR)$ with $u_0,v_0\ge 0$, if $\rho(u_0,v_0)$ is defined, then for any $t_0\in\RR$,
$\rho(u(t+t_0,\cdot;t_0,u_0),u(t+t_0,\cdot;t_0,v_0))$ is also defined for all $t>0$. Moreover,
$\rho(u(t+t_0,\cdot;t_0,u_0),u(t+t_0,\cdot;t_0,v_0))$  is non-increasing in $t$.

\item[(2)] For any $\epsilon>0$, $\sigma>0$, $M>0$,  and $\tau>0$  with $\epsilon<M$ and
$\sigma\le \ln \frac{M}{\epsilon}$, there is $\delta>0$   such that
for any $u_0,v_0\in C_{\rm unif}^b(\RR)$ with $\epsilon\le u_0(x)\le M$, $\epsilon\le v_0(x)\le M$ for $x\in\RR$ and
$\rho(u_0,v_0)\ge\sigma$, there holds
$$
\rho(u(\tau+t_0,\cdot;t_0,u_0),u(\tau+t_0,\cdot;t_0,v_0))\le \rho(u_0,v_0)- \delta\quad \text{for all}\quad t_0\in\RR.
$$
\end{itemize}\end{proposition}

\begin{proof} It follows from the similar arguments as  those in \cite[Proposition 3.4]{KoSh}.
\end{proof}

\begin{proof}[Proof of Theorem \ref{positive-solu-thm}]
We divide the proof into two steps.

\smallskip

\noindent {\bf Step 1.} In this step,  we prove the existence, uniqueness, and stability of  bounded strictly positive entire solutions $u^+(t,x)$ with
$u^+(t,x)$ being uniformly continuous in $(t,x)\in\RR\times \RR$.

Note that the existence, uniqueness, and stability of bounded strictly positive entire solutions  follows from the similar arguments as those in \cite[Theorem 1.1(1)]{CaSh}.
We  outline the  proof of existence in the following  for the use in the proof of uniform continuity.

Clearly,  $u(t,x)\equiv M$ is a super-solution of \eqref{main-eq} for any $M\gg 1$. For given $\delta>0$, let
$v_\delta\equiv \delta$. By the similar  arguments of  \cite[Theorem 1.1(1)]{CaSh}, there are $\delta_0>0$ and $T>0$ such that
for any $0<\delta\le\delta_0$,
\begin{equation}
\label{aux-pary-eq2}
u(t_0+T,\cdot;t_0,v_\delta)\ge v_\delta\quad\forall\,\, t_0\in\RR.
\end{equation}
Fix $M\gg 1$ and $0<\delta\ll 1$. Let $u_M\equiv M$ and $u_\delta\equiv \delta$. For given $m\in\ZZ$ and  $n\in\ZZ^+$, let
$u_{m-n,m}(\cdot)=u(mT,\cdot;(m-n)T,u_\delta)$ and $u^{m-n,m}(\cdot)=u(mT,\cdot;(m-n)T,u_M)$. Then for any $m\in\ZZ$ and
$n\ge 0$,
$$
\delta\le  u_{m-1,m}(\cdot)\le u_{m-2,m}(\cdot)\le \cdots,\,\, M\ge u^{m-1,m}(\cdot)\ge u^{m-2,m}(\cdot)\ge \cdots.
$$
Moreover, by Proposition \ref{part-metric-prop}, it is not difficult to prove that
\begin{equation}
 \label{continuity-eq0}
 \rho(u^{m-n,m},u_{m-n,m})\to 0\quad {\rm as}\,\, n\to\infty\quad \text{uniformly in}\,\, m\in\ZZ.
 \end{equation}
 Hence there are $u^{*,m}\in C_{\rm unif}^b(\RR)$ ($m\in\ZZ$) such that
$$
u^{*,m}(x)=\lim_{n\to\infty} u_{m-n,m}(x)=\lim_{n\to\infty} u^{m-n,m}(x)\quad \text{uniformly in}\,\, x\in\RR.
$$
It is then not difficult to see that
 $u(t,\cdot;0,u^{*,0})$ has backward extension for all $t<0$ and hence
 $u^+(t,x):=u(t,\cdot;0,u^{*,0})$ is an entire solution.  Moreover,
 \begin{equation}
 \label{continuity-eq3}
 \delta\le u^+(kT,x)\le M\quad \forall\,\,  x\in\RR,\,\, k\in\ZZ
 \end{equation}
 and
 \begin{equation}
 \label{continuity-eq4}
 u^+(kT,x)=\lim_{n\to\infty}u(kT,x;(k-n)T,u_M)\quad \text{uniformly in}\,\, x\in\RR,\,\, k\in\ZZ.
 \end{equation}
By \eqref{continuity-eq3}, $u^+(t,x)$ is a bounded strictly positive entire solution of \eqref{main-eq}.

 We now prove the uniform continuity of $u^+(t,x)$ in $(t,x)\in\RR\times\RR$.
 In the case that $\mathcal{A}u=u_{xx}$, by the regularity and
a priori estimates for parabolic equations, we have that $u^+(t,x)$ is uniformly continuous in $(t,x)\in\RR\times\RR$.
We then only need to prove that $u^+(t,x)$ is uniformly continuous in $(t,x)\in\RR\times\RR$ in the case that
 $\mathcal{A}$ is a nonlocal dispersal operator. In this case, by the boundedness of $u^+(t,x)$, $u_t^+(t,x)$ is uniformly bounded.
 This implies that $u^+(t,x)$ is uniformly continuous in $t$ uniformly with respect to $x\in\RR$. We claim that
 $u^+(t,x)$ is also uniformly continuous in $x$ uniformly with respect to $t\in\RR$. Indeed,
 By \eqref{continuity-eq4}, for any $\epsilon>0$, there is $K\in \NN$ such that
 \begin{equation}
 \label{continuity-eq5}
 |u^+(t+kT,x)-u(t+kT,x;(k-K)T,u_M)|<\epsilon\quad \forall\, \, x\in\RR,\,\, 0\le t\le T,\,\, k\in\ZZ.
 \end{equation}
 It then suffices to prove that $u(t+kT,x;(k-K)T,u_M)$ is uniformly continuous in $x$ uniformly with respect to $t\in [0,T]$ and
 $k\in\ZZ$. For any given $h>0$ and $k\in\ZZ$, let
 $$
 v(t,x;h,k)=u(t,x+h;(k-K)T,u_M)-u(t,x;(k-K)T,u_M).
 $$
 Then $v(t,x;h,k)$ satisfies
 $$
 \begin{cases}
 v_t=\mathcal{A}v(t,x)+p(t,x)v(t,x;h,k)+q(t,x),\quad x\in\RR,\,\, t>(k-K)T\cr
 v((k-K)T,x;h,k)=0,\quad x\in\RR,
 \end{cases}
 $$
 where
 \begin{align*}
 p(t,x)=&f(t,x+h,u(t,x+h;(k-K)T,u_M))+u(t,x;(k-K)T,u_M)\cdot\\
 &\quad \int_0^1 f_u(t,x,su(t,x+h;(k-K)T,u_M)+(1-s)u(t,x;(k-K)T,u_M))ds
 \end{align*}
 and
 $$
 q(t,x)=u(t,x;(k-K)T,u_M)\big[ f(t,x+h,u(t,x+h;(k-K)T,u_M))-f(t,x,u(t,x+h;(k-K)T,u_M))\big].
 $$
 Then
 $$
 v(t,\cdot;h,k)=\int_{(k-K)T}^t e^{\mathcal{A}(t-\tau) }p(\tau,\cdot) v(\tau,\cdot;h,k)d\tau+\int_{(k-K)T}^t e^{\mathcal{A}(t-\tau)}q(\tau,\cdot)d\tau.
 $$
 This together with the assumption (H0) implies that there is $C>0$ such that
 $$
 |v(t,x;h,k)|\le C h\quad \forall \,\, 0\le t\le T,\,\, x\in\RR,\,\, k\in\ZZ.
 $$
 Then by \eqref{continuity-eq5} and Gronwall's inequality,  $u^+(t,x)$ is uniformly continuous in $x$ uniformly with respect to $x\in\RR$ and
 then $u^+(t,x)$ is uniformly continuous in $(t,x)\in\RR\times\RR$.

By the similar arguments as those in \cite[Theorem 1.1(1)]{CaSh}, we have that  bounded strictly positive entire solutions
 of \eqref{main-eq} are stable and unique. This completes the proof of Step 1.

\smallskip

\noindent {\bf Step 2.} In this step, we show that
 $u^+(t,x)$ is almost periodic in $t$ (resp., in $x$) if $f(t,x,u)$ is almost periodic in $t$ (resp., in $x$).

 Assume that $f(t,x,u)$ is almost periodic in $t$.
For any given sequences $\{\alpha^{'}_n\}$, $\{\beta^{'}_n\}\subset \RR$, there are subsequences
$\{\alpha_n\}\subset\{\alpha_n^{'}\}$, $\{\beta_n\}\subset\{\beta_n^{'}\}$ such that
$$
\lim_{m\to\infty}\lim_{n\to\infty}f(t+\alpha_n+\beta_m,x,u)=\lim_{n\to\infty} f(t+\alpha_n+\beta_n,x,u).
$$
By the uniform continuity of $u^+(t,x)$, without loss of generality, we may assume that
$\lim_{n\to\infty}u^+(t+\alpha_n,x)$ exists locally uniformly in $(t,x)\in\RR\times\RR$. Let
$$
\hat f(t,x,u)=\lim_{n\to\infty}f(t+\alpha_n,x,u)\quad {\rm and}\quad
\hat u^+(t,x)=\lim_{n\to\infty} u^+(t+\alpha_n,x).
$$
It is clear that $\hat f(t,x,u)$ satisfies (H0) and  $\hat u^+(t,x)$ is uniformly continuous in $(t,x)\in\RR\times\RR$
and $\inf_{t\in\RR,x\in\RR}\hat u^+(t,x)>0$. Moreover,
$\hat u^+(t,x)$ is a bounded positive entire solution of \eqref{main-eq} with $f$ being replaced by $\hat f$.
Similarly, without loss of generality, we may assume that
$\lim_{m\to\infty}\hat u^+(t+\beta_m,x)$ and $\lim_{n\to\infty} u^+(t+\alpha_n+\beta_n,x)$ exist locally uniformly
in $(t,x)\in\RR\times\RR$. Let
$$
\check f(t,x,u)=\lim_{m\to\infty} \hat f(t+\beta_m,x,u),\quad \tilde f(t,x,u)=\lim_{n\to\infty}f(t+\alpha_n+\beta_n,x,u)
$$
and
$$
\check u^+(t,x)=\lim_{m\to\infty}\hat u^+(t+\beta_m,x,u),\quad \tilde u^+(t,x)=\lim_{n\to\infty} u^+(t+\alpha_n+\beta_n,x).
$$
Then $\check f(t,x,u)$ and $\tilde f(t,x,u)$ satisfy (H0), $\check u^+(t,x)$ and $\tilde u^+(t,x)$ are uniformly continuous in $(t,x)\in\RR\times\RR$,
$\check u^+(t,x)$ is a bounded positive entire solution of \eqref{main-eq} with $f$ being replaced by $\check f$, and
$\tilde u^+(t,x)$ is a bounded positive entire solution of \eqref{main-eq} with $f$ being replaced by $\tilde f$.
Note that $\tilde f=\check f$. Then by the uniqueness of positive entire solutions of \eqref{main-eq} with $f$ being replaced by $\check f$, we have
$$
\check u^+(t,x)=\tilde u^+(t,x).
$$
This together with Remark \ref{almost-periodic-rk} implies that $u^+(t,x)$ is almost periodic in $t$ uniformly with respect to $x\in\RR$.
In particular, if $f$ is periodic in $t$, so is $u^+(t,x)$.

Assume that $f(t,x,u)$ is almost periodic in $x$. Similarly, we can proved that $u^+(t,x)$ is almost periodic in $x$ uniformly with respect to $t\in\RR$.
In particular, if $f$ is periodic in $x$, so is $u^+(t,x)$. The theorem is thus proved.
\end{proof}

\section{Stability of transition waves}

In this section, we study the stability of transition waves in the general case and prove Theorem \ref{main-thm-general-case}.

\begin{proof}[Proof of Theorem \ref{main-thm-general-case}]
Suppose that $u=U(t,x)$ is a transition wave of \eqref{main-eq} satisfying \eqref{cond-eq-2} and \eqref{cond-eq1}.

Note that, for given $u_0(\cdot)$ satisfying \eqref{cond-eq3} and given $t_0\in\RR$,
the part metric $\rho(u_0,U(t_0))$ is well defined and then
$\rho(u(t,\cdot;t_0,u_0),U(t+t_0,\cdot))$ is well defined for all $t\ge 0$. To prove \eqref{cond-eq4}, it suffices to prove that
for any $\epsilon>0$,  there is $T>0$ such that
\begin{equation}
\label{proof-thm1-eq1}
\rho(u(t+t_0,\cdot;t_0,u_0),U(t+t_0,\cdot))<\epsilon\quad \text{for
all}\,\, t\ge T.
\end{equation}
In fact, by Proposition \ref{part-metric-prop}(1),  we only need to prove that for any $\epsilon>0$,
\begin{equation}
\label{proof-thm1-eq2}
\rho(u(t+t_0,\cdot;t_0,u_0),U(t+t_0,\cdot))<\epsilon\quad \text{for
some}\,\, t>0.
\end{equation}

Assume by contradiction that there is $\epsilon_0>0$ such that
\begin{equation}
\label{eqq-0} \rho(u(t+t_0,\cdot;t_0,u_0),U(t+t_0,\cdot))\ge
\epsilon_0
\end{equation}
for all $t\ge 0$.
Fix a $\tau>0$. We claim that  if \eqref{eqq-0} holds, then there is $\delta>0$ such that
\begin{align}
\label{main-thm-eq1}
\rho(u(\tau+s+t_0,\cdot;t_0,u_0),U(\tau+s+t_0,\cdot))\le
\rho(u(s+t_0,\cdot;t_0,u_0),U(s+t_0,\cdot))-\delta
\end{align}
for all $s\ge 0$.

Before proving \eqref{main-thm-eq1}, we prove that \eqref{main-thm-eq1} gives rise to a contradiction.
In fact, assume \eqref{main-thm-eq1}. Then  we have
$$
\rho(u(n\tau+t_0,\cdot;t_0,u_0),U(n\tau+t_0,\cdot))\le
\rho(u(t_0,\cdot;t_0,u_0),U(t_0,\cdot))-n\delta
$$
for all $n\ge 0$. Letting $n\to\infty$, we have  $\rho(u(n\tau+t_0,\cdot;t_0,u_0),U(n\tau+t_0,\cdot))\to -\infty$, which is
a contradiction. Therefore, \eqref{eqq-0} does not hold and then for any $\epsilon>0$,
$$
\rho(u(t+t_0,\cdot;t_0,u_0),U(t+t_0,\cdot))<\epsilon\quad \text{for
some}\,\, t>0.
$$
This together with  Proposition \ref{part-metric-prop} implies that
$$
\lim_{t\to\infty}\rho(u(t+t_0,\cdot;t_0,u_0),U(t+t_0,\cdot))=0.
$$
The theorem then follows.

We now prove that if \eqref{eqq-0} holds, then there is $\delta>0$ such that  \eqref{main-thm-eq1} holds.

First, for any $0<\epsilon<\frac{\epsilon_0}{4+2\epsilon_0}$, by \eqref{cond-eq3}, \eqref{cond-eq-1},  and \eqref{cond-eq0},
there is $d_1\gg d_1^*$ such that
\begin{align*}
d^*(1-\epsilon) \phi(t_0,x)-d_1 \phi_1(t_0,x)\le u_0(x)\le d^*(1+\epsilon) \phi(t_0,x)+d_1 \phi_1(t_0,x).
\end{align*}
By \eqref{cond-eq2}, there holds
\begin{align*}
d^*(1-\epsilon) \phi(t,x)-d_1 \phi_1(t,x)\le u(t,x;t_0,u_0)
\le d^* (1+\epsilon)\phi(t,x)+d_1 \phi_1(t_0,x)\quad \forall\,\, t\ge t_0.
\end{align*}
{Then by \eqref{cond-eq0} and \eqref{cond-eq1},  for  $x-X(t)\gg 1$, we have that $0<\frac{\phi_1(t,x)}{\phi(t,x)}\ll 1$,
\begin{align*}
u(t,x;t_0,u_0)&\le d^*(1+\epsilon)\phi(t,x)\Big(1+\frac{d_1}{d^*(1+\epsilon)}\frac{\phi_1(t,x)}{\phi(t,x)}\Big)\\
&\le(1+\epsilon) U(t,x)\Big(1+\frac{d_1}{d^*(1+\epsilon)}\frac{\phi_1(t,x)}{\phi(t,x)}\Big)\Big(1-\frac{d_1^*}{d^*}\frac{\phi_1(t,x)}{\phi(t,x)}\Big)^{-1},
\end{align*}
and
\begin{align*}
u(t,x;t_0,u_0)&\ge  d^*(1-\epsilon)\phi(t,x)\Big(1-\frac{d_1}{d^*(1-\epsilon)}\frac{\phi_1(t,x)}{\phi(t,x)}\Big)\\
&\ge (1-\epsilon) U(t,x)\Big(1-\frac{d_1}{d^*(1-\epsilon)}\frac{\phi_1(t,x)}{\phi(t,x)}\Big)\Big(1+\frac{d_1^*}{d^*}\frac{\phi_1(t,x)}{\phi(t,x)}\Big)^{-1}.
\end{align*}}
This together with \eqref{cond-eq-2} and \eqref{cond-eq0} implies that,  for any $s\ge 0$, there is $x_s(\ge X(t_0+s))$
 such that
 \begin{equation}
 \label{eqq-00}
 \sup_{s\ge 0, t\in [t_0+s,t_0+s+\tau]}|x_s-X(t)|<\infty
 \end{equation}
 and
\begin{align}
\label{eqq-1}
\frac{1}{1+\epsilon_0/2}U(t,x)\le u(t,x;t_0,u_0)\le (1+\epsilon_0/2)U(t,x)\quad \forall\,\, t\in [t_0+s,t_0+s+\tau],\,\, x\ge x_s.
\end{align}


Next, we claim that
\begin{equation}
\label{eqq-5-0}
\inf_{s\ge 0, t\in[s+t_0,\tau+s+t_0],x\le x_s} U(t,x)>0.
\end{equation}
In fact, if this is not true, then there are $s_n\ge 0$, $t_n\in [s_n+t_0,\tau+s_n+t_0]$, and $x_n\le x_{s_n}$ such that
$$
\lim_{n\to\infty}U(t_n,x_n)=0.
$$
By Remark \ref{transition-wave-rk}(1) and \eqref{eqq-00}, there are $\beta>0$ and  $\tilde x_n\le x_n$ such that
$$
U(s_n+t_0,\tilde x_n)\ge \beta
\quad {\rm and}\quad  \sup_{n\ge 1} |x_{s_n}-\tilde x_n|<\infty.
$$
Let
$$
U_n(t,x)=U(t+s_n+t_0,x+\tilde x_{n}).
$$
By the uniform continuity of $U(t,x)$ in $(t,x)\in\RR\times\RR$, without loss of generality, we may assume that
$$
\lim_{n\to\infty}U_n(t,x)=U^*(t,x)
$$
locally uniformly in $(t,x)\in\RR\times\RR$.
Without loss of generality, we may also assume that
$$
\lim_{n\to\infty}\big(t_n-s_n-t_0\big)=t^*,\quad \lim_{n\to\infty }\big(x_n-\tilde x_{n}\big)=x^*,
$$
and
$$
\lim_{n\to\infty} f(t+s_n+t_0,x+\tilde x_{n},u)=f^*(t,x,u)
$$
locally uniformly in $(t,x,u)\in\RR\times\RR\times\RR$.
Then
\begin{equation}
\label{eqq-5-1}
U^*(t,x)\ge 0, \quad
U^*(t^*,0)\ge \beta,\quad U^*(t^*,x^*)=0,
\end{equation}
and  $U^*(t,x)$ is a solution of \eqref{main-eq} with $f(t,x,u)$ being replaced by $f^*(t,x,u)$. Then by comparison
principle, we have either $U^*(t,x)\equiv 0$ or $U^*(t,x)>0$ for all $(t,x)\in\RR\times\RR$,
which contradicts to \eqref{eqq-5-1}. Therefore, \eqref{eqq-5-0} holds.

Now for given $s\ge 0$,
let
$$
\rho(s+t_0)=\rho(u(s+t_0,\cdot;t_0,u_0),U(s+t_0,\cdot)).
$$
By \eqref{eqq-0},
\begin{equation}
\label{eqq-3-0} \rho(t_0)\ge \rho(s+t_0)\ge \epsilon_0
\end{equation}
and
\begin{align}
\label{eqq-3}
\frac{1}{\rho(s+t_0)}U(s+t_0,\cdot)&\le u(s+t_0,\cdot;t_0,u_0)\le \rho(s+t_0) U(s+t_0,\cdot).
\end{align}
It follows from \eqref{eqq-3} and Proposition \ref{comparison-prop}  that
$$
u(t+s+t_0,\cdot;t_0,u_0)\le u(t+s+t_0,\cdot;s+t_0,\rho(s+t_0)U(s+t_0,\cdot))\quad {\rm for}\quad  t\ge 0.
$$

Let
$$
\hat u(t,x)= u(t+s+t_0,\cdot;s+t_0,\rho(s+t_0)U(s+t_0,\cdot)),
$$
$$\tilde u(t,x)=\rho(s+t_0) u(t+s+t_0,\cdot;s+t_0,U(s+t_0,\cdot))\big(=\rho(s+t_0) U(t+s+t_0,x)\big),
$$
and
$$
\bar u(t,x)=\tilde u(t,x)-\hat u(t,x).
$$
Then
\begin{align}
\label{eqq-4}
\bar u_t(t,x)&=\mathcal{A}\bar u(t,x)+\tilde u(t,x) f(t+s+t_0,x,U(t+s+t_0,x))-\hat u(t,x) f(t+s+t_0,x,\hat u(t,x))\nonumber\\
&=\mathcal{A}\bar u(t,x)+p(t,x) \bar u(t,x)+b(t,x),
\end{align}
where
$$
p(t,x)=f(t+s+t_0,x,\hat u(t,x))+\tilde  u(t,x)\int_0^1 f_u(t+s+t_0,x,r\tilde u(t,s)+(1-r)\hat u(t,x))dr,
$$
and
\begin{align*}
b(t,x)&=\tilde u(t,x) \big[f(t+s+t_0,x,U(t+s+t_0,x))-f(t+s+t_0,x,\tilde u(t,x))\big]\\
&= \tilde  u(t,x)\big[f(t+s+t_0,x,U(t+s+t_0,x))-f(t+s+t_0,x,\rho(s+t_0) U(t+s+t_0,x))\big].
\end{align*}
By (H0) and \eqref{eqq-5-0}, there is $b_0>0$ such that for any $s\ge 0$,
\begin{equation}
\label{eqq-5}
\inf_{t\in [s+t_0,\tau+s+t_0],x\le x_s} b(t,x)\ge b_0>0.
\end{equation}

In the case that $\mathcal{A}u=u_{xx}$,
by \eqref{eqq-4} and comparison principle for parabolic equations, we have
$$
\bar u(\tau,\cdot)\ge \int_{t_0+s}^{t_0+s+\tau} e^{(\tau+s+t_0-r) p_{\inf}} T(\tau+s+t_0 -r) b(r,\cdot) dr,
$$
where $p_{\inf}=\inf_{t\in [t_0+s,t_+s+\tau,x\in\RR}p(t,x)$,
$$
T(t)u(x)=\int_{\R^{N}}G(x-y,t)u(y)dy
$$
and
\begin{equation}
\label{heat-kernel}
G(x,t)=\frac{1}{(4\pi t)^{\frac{N}{2}}}e^{-\frac{|x|^{2}}{4t}}.
\end{equation}
This together with \eqref{eqq-5} implies that there is $\delta_0>0$ such that for any $s\ge 0$,
\begin{equation*}
\bar u(\tau,x)\ge \delta_0\quad \forall \,\, x\le x_s
\end{equation*}
and then
\begin{equation}
\label{eqq-6}
u(\tau+s+t_0,x;t_0,u_0)\le \rho(s+t_0)U(t+s+t_0,x)-\delta_0\quad \forall\,\, x\le x_s.
\end{equation}

By \eqref{eqq-1} and \eqref{eqq-6},
then there is $\delta_1>0$ such that
$$
u(\tau+s+t_0,\cdot;t_0,u_0)\le (\rho(s+t_0)-\delta_1) U(\tau+s+t_0,\cdot)\quad \text{for all}\,\, s\ge 0.
$$
Similarly, we can prove that there is $\delta_2>0$ such that
$$
\frac{1}{\rho(s+t_0)-\delta_2}U(\tau+s+t_0,\cdot)\le u(\tau+s+t_0,\cdot;t_0,u_0)\quad \text{for all}\,\, s\ge 0.
$$
The claim \eqref{main-thm-eq1} then holds for
 $\delta=\min\{\delta_1,\delta_2\}$.

In the case that $\mathcal{A}u(t,x)=\int_{\RR}\kappa(y-x)u(t,y)dy-u(t,x)$,  we have that $\bar u(t,x)$ satisfies
$$
\bar u_t=\int_{\RR}\kappa(y-x)\bar u(t,y)dy-\bar u(t,x)+p(t,x)\bar u(t,x)+b(t,x)\quad \forall\,\, x\in\RR.
$$
Note that
$$\int_{\RR}\kappa(y-x)\bar u(t+s+t_0,y)dy\ge 0.
$$
Hence for $x\le x_s$,
$$
\bar u_t(t,x)\ge (-1+ p(t,x))\bar u(t,x)+b_0.
$$
This implies that
$$
\bar u(\tau,x)\ge \int_{s+t_0}^{\tau+s+t_0} e^{\big(-1+p_{\inf}\big)(\tau+s+t_0-r)}b_0 dr\quad \forall\,\, x\le x_s.
$$
Then by the similar arguments as in the above,
the claim \eqref{main-thm-eq1} then holds for
 some $\delta>0$.
\end{proof}

\section{Existence and stability of transition waves}

In this section, we consider the stability of transition waves of \eqref{main-eq}
established in literature  and prove Theorem \ref{main-thm-special-case}.

\begin{proof}[Proof of Theorem \ref{main-thm-special-case} (1)] As it is mentioned in Remark \ref{special-case-rk1}(1),
 the existence of transition waves is established in \cite{NaRo2}.
In the following, we outline the construction of transition waves from \cite{NaRo2} and  show that they satisfy the conditions in Theorem \ref{main-thm-general-case} and hence are asymptotically stable.

 Assume (H1) and let $a(t,x)=f(t,x,0)$.
For any $\mu>0$, by \cite[Lemma 3.1]{NaRo2}, the equation
$$
u_t=u_{xx}+a(t,x)u
$$
has a positive solution of the form
$$
u_\mu(t,x)=e^{-\mu x}\eta_\mu(t,x),\quad {\rm where}\quad \eta_\lambda(t,x+p)=\eta_\lambda(t,x).
$$
By Lemma \cite[Lemma 3.2]{NaRo2}, there are $\beta>0$ and a uniformly Lipschitz continuous function $S_\mu(t)$ such that
$$
|S_\mu(t)-\frac{1}{\mu}\ln \|\eta_\mu(\cdot,t)\|_{L^\infty}|\le \beta \quad \forall\,\, t\in\RR.
$$
Let
$$
c_\mu=\liminf_{t-s\to\infty}\frac{S_\mu(t)-S_\mu(s)}{t-s}.
$$
By \cite[Lemma 3.4]{NaRo2},
there is $\mu^*>0$ such that for $0<\mu<\mu^*$, $c_\mu$ is decreasing for $\mu\in (0,\mu^*)$ and this does not hold
for $\mu \in (0,\tilde\mu^*)$ for any $\tilde\mu^*>\mu^*$.
Let $c^*=c_{\mu^*}$. Note that when $a(x)\equiv a$, $c^*=2\sqrt a$.
We show that Theorem \ref{main-thm-special-case}(1) holds with this $c^*$.

To this end,  for fixed $\mu>0$,
let
$$
\tilde \phi_\mu(t,x)=e^{-\mu S_\mu(t)}\eta_\mu(t,x).
$$
By equation (38) in \cite{NaRo2},  there is $C_\mu>0$ such that
$$
C_\mu\le \tilde \phi_\mu(t,x)\le e^{\mu \beta}\quad \forall\,\, x\in\RR,\,\, t\in\RR.
$$
Note that for any $c>c^*$, there is $0<\mu<\mu^*$ such that $c_\mu=c$. Let $\mu^{'}$ be such that
$\mu<\mu^{'}<(1+\nu)\mu$. By \cite[Lemma 3.2]{NaRo1}, there is $\sigma(\cdot)\in W^{1,\infty}(\RR)$ such that
$$
\inf \{ \sigma^{'}(t)+\mu^{'}(c_\mu(t)-c_{\mu^{'}}(t)\}>0,
$$
where $c_\mu(t)=S^{'}_\mu(t)$ and $c_{\mu^{'}}(t)=S_{\mu^{'}}^{'}(t)$.
Let
$$
\phi(t,x)=e^{-\mu(x-S_\mu(t)}\tilde \phi_\mu(t,x)\quad {\rm and}\quad
\phi_1(t,x)=e^{\sigma(t)-\mu^{'}(x-S_\mu(t))}\tilde \phi_{\mu^{'}}(t,x).
$$
By the arguments of \cite[Theorem 2.1]{NaRo2},  for any $u_0\in C_{\rm unif}^b(\RR)$ $(u_0(x)\ge 0)$ and $t_0\in\RR$ with
$$
 u_0(x)\le d\phi(t_0,x)+d_1 \phi_1(t_0,x) \quad \big({\rm resp.}\,\, u_0(x)\ge  d\phi(t_0,x)-d_1 \phi_1(t_0,x)\big)\quad \forall\,\, x\in\RR
$$
for some $d>0$ and $d_1>0$ (resp., for some $d>0$ and $d_1\gg 1$), there holds
$$
u(t,x;t_0,u_0)\le d\phi(t,x)+d_1 \phi_1(t,x) \quad \big({\rm resp.}\,\, u(t,x;t_0,u_0)\ge  d\phi(t,x)-d_1 \phi_1(t,x)\big)\quad \forall\,\, x\in\RR
$$
for all $t\ge t_0$.
Moreover, there is a transition wave solution
$u=U_\mu(t,x)$ of \eqref{main-eq} satisfying that
$$
\phi(t,x)-d_\mu \phi_1(t,x)\le U_\mu(t,x)\le \phi(t,x)\quad \forall\,\, t,x\in\RR
$$
for some $d_\mu\gg 1$.
Clearly, $X(t)=S_\mu(t)$ is an interface location of $u=U_\mu(t,x)$ satisfying \eqref{cond-eq-2}, and $\phi(t,x)$ and $\phi_1(t,x)$ satisfy \eqref{cond-eq-1} and \eqref{cond-eq0}.
 It then follows from Theorem \ref{main-thm-general-case} that $u=U_\lambda(t,x)$ is asymptotically stable.  Clearly, $u=U_{\mu}(t,x)$ has least mean speed $c$. Theorem \ref{main-thm-special-case}(1) thus follows.
\end{proof}

\begin{proof}[Proof of Theorem \ref{main-thm-special-case} (2)]
 As it is mentioned in Remark \ref{special-case-rk1}(2),
 the existence of transition waves is established in \cite{NaRo3}.
 Similarly, we outline the construction of transition waves from \cite{NaRo3} and  show that they satisfy the conditions in Theorem \ref{main-thm-general-case} and hence are asymptotically stable.

Assume (H2) and let $a(x)=f(x,0)$.
By \cite[Proposition 1.3]{NaRo3},  for any $\lambda>\lambda_0$, there is a unique positive $\phi_\lambda\in C^2(\RR)$ such that
$$
\phi^{''}_\lambda+a(x)\phi_\lambda(x)=\lambda\phi_\lambda(x) \quad {\rm in}\,\,\RR, \,\,\, \phi_\lambda(0)=1,\,\, \lim_{x\to\infty}\phi_\lambda(x)=0,
$$
and there exists the limit
\begin{equation}
\label{ap-eq1}
\mu(\lambda):=-\lim_{x\to \pm \infty}\frac{1}{x}\ln \phi_\lambda(x)>0.
\end{equation}
By \cite[Lemma 2.3]{NaRo3}, $\phi_\lambda(x)$ is unbounded, and by \cite[Lemma 2.4]{NaRo3},
$\phi^{'}_\lambda(x)/\phi_\lambda(x)$ is almost periodic in $x$.
Let
$$
c^*=\inf_{\lambda>\lambda_0}\frac{\lambda}{\mu(\lambda)}.
$$
In the following we show that Theorem \ref{main-thm-special-case}(2) holds with this $c^*$.

By \cite[Lemma 3.2]{NaRo3}, for any $c>c^*$, there is $\lambda>\lambda_0$ such that
$$
c=\frac{\lambda}{\mu(\lambda)}\quad {\rm and}\quad c>\frac{\lambda^{'}}{\mu(\lambda^{'})}\quad {\rm for}\,\,\, \lambda^{'}-\lambda>0\,\,\text{small enough}.
$$
Let $\sigma_\lambda(x)=-\frac{\phi_\lambda^{'}(x)}{\phi_\lambda(x)}$. By \cite[Proposition 3.3]{NaRo3}, there exist $\delta>0$, $\epsilon\in (0,1)$,
and a function $\theta\in C^2(\RR)\cap L^\infty(\RR)$ such that
$$
\inf_{x\in\RR} \theta(x)>0,\quad -\theta^{''}+2\sigma_\lambda\phi^{'}-(\sigma^2_\lambda-\sigma^{'}_\lambda+a)\theta\ge (\delta-(1+\epsilon)\lambda)\theta\quad {\rm in}\,\,\RR.
$$
By the arguments of \cite[Proposiiton 3.4]{NaRo3}, for any $t_0\in\RR$ and $u_0\in C_{\rm unif}^b(\RR)$ ($u_0\ge 0$) with
$$
u_0(x)\le d\phi_\lambda(x) e^{\lambda t_0}+ d_1 \theta(x) \phi_\lambda^{1+\epsilon}e^{(1+\epsilon)\lambda t_0}$$
$$
\big({\rm resp.},
u_0(x)\ge d\phi_\lambda(x) e^{\lambda t_0}- d_1 \theta(x) \phi_\lambda^{1+\epsilon}e^{(1+\epsilon)\lambda t_0}\big)
$$
for some $d>0$ and $d_1>0$ (resp., for $d>0$ and $d_1\gg 1$),
 there holds
 $$
u(t,x;t_0,u_0)\le d\phi_\lambda(x) e^{\lambda t}+ d_1 \theta(x) \phi_\lambda^{1+\epsilon}e^{(1+\epsilon)\lambda t}
$$
$$
\big({\rm resp.},
u(t,x;t_0,u_0)\ge d\phi_\lambda(x) e^{\lambda t}- d_1 \theta(x) \phi_\lambda^{1+\epsilon}e^{(1+\epsilon)\lambda t}\big)
$$
 for all $t\ge t_0$.
 Moreover, there is a transition wave $u(t,x)=U_\lambda(t,x)$ of \eqref{main-eq} with mean speed $c$ and
satisfying that
$$
\phi_\lambda(x)e^{\lambda t}-d_1 \theta(x) \phi_\lambda^{1+\epsilon}e^{(1+\epsilon)\lambda t}\le U_\lambda(t,x)\le \phi_\lambda(x) e^{\lambda t}
$$
for some $d_1\gg 1$.

Let $X(t)$ be such that $\phi_\lambda(X(t))=e^{-\lambda t}$, that is,
$\int_0^{X(t)}\frac{\phi_\lambda^{'}(y)}{\phi_\lambda(y)}dy=-\lambda t$. By \cite[Lemma 3.5]{NaRo3}, $X(t)$ is an interface location of $u=U_\lambda(t,x)$.
By \cite[Proposition 4.2]{HaRos}, $X(t)$ satisfies \eqref{cond-eq-2}.
Let
$$
\phi(t,x)=\phi_\lambda(x)e^{\lambda t},\quad \phi_1(t,x)=\theta(x) \phi_\lambda^{1+\epsilon}e^{(1+\epsilon)\lambda t}.
$$
It is clear  that $\phi(t,x)$ and $\phi_1(t,x)$ satisfy \eqref{cond-eq-1}.  Note that
\begin{align}
\label{ap-eq2}
\frac{\phi_1(t,x+X(t))}{\phi(t,x+X(t))}=\frac{\Big(e^{\int_0^{x+X(t)}\frac{\phi^{'}_\lambda(y)}{\phi_\lambda(y)}dy}e^{\lambda t}\Big)^{1+\epsilon}}{e^{\int_0^{x+X(t)}\frac{\phi^{'}_\lambda(y)}{\phi_\lambda(y)}dy}e^{\lambda t}}=e^{\epsilon\int_{X(t)}^{x+X(t)}\frac{\phi^{'}_\lambda(y)}{\phi_\lambda(y)}dy}.
\end{align}
By the almost periodicity of $\frac{\phi^{'}_\lambda(x)}{\phi_\lambda(x)}$, \eqref{ap-eq1}, and \eqref{ap-eq2},
$$
\lim_{|x|\to\infty}\frac{1}{x}\int_{X(t)}^{x+X(t)}\frac{\phi^{'}_\lambda(y)}{\phi_\lambda(y)}dy =-\mu(\lambda)<0
$$
uniformly in $t\in\RR$. This implies that $\phi_1(t,x)$ and $\phi_2(t,x)$ satisfy \eqref{cond-eq0}.
It then follows from Theorem \ref{main-thm-general-case} that $u=U_\lambda(t,x)$ is an asymptotically stable transition wave of \eqref{main-eq}
with average speed $c$.
\end{proof}

\begin{proof} [Proof of Theorem \ref{main-thm-special-case}(3)]  The existence of transition waves is established in \cite{Zla}.
In the following, we outline the construction of transition waves from \cite{Zla} and  show that they are asymptotically stable by using the arguments
in the proof of  Theorem \ref{main-thm-general-case}.

Recall that
$$
\lambda_0=\sup\sigma[\p_{xx}^2+a(\cdot)],\quad \lambda_1=2a_-.
$$
Note that $\lambda_0\ge a_-$. By the arguments of \cite[Theorem 1.1]{Zla}, for any $\lambda>\lambda_0$, there is
a unique $\phi_\lambda(x)$ such that
$$
\phi_\lambda^{''}(x)+a(x)\phi_\lambda(x)=\lambda\phi_\lambda(x)\quad {\rm for}\quad x\in\RR
$$
and
$$
\phi_\lambda(0)=1,\quad \lim_{x\to\infty}\phi_\lambda(x)=0.
$$

Fix $\lambda\in (\lambda_0,\lambda_1)$ and  $1-\frac{\lambda_1-\lambda}{a_+}<\alpha<1$. Let $U_{g,\sqrt \alpha}(x)$ be the traveling front profile for the PDE
$$
u_t=u_{xx}+g(u)
$$
with propagating speed $c_{1,\sqrt a}=\sqrt \alpha+\frac{1}{\sqrt \alpha}> 2$ and
$$
\lim_{x\to\infty} U_{g,\sqrt \alpha}(x) e^{\sqrt \alpha x}=1.
$$
 Let
$$
h_{g,\alpha}(v)=U_{g,\sqrt a}(-\alpha^{-\frac{1}{2}}\ln v)
$$
for $v>0$ and $h_{g,\sqrt \alpha}(0)=0$. Then
$h^{'}_{g,\sqrt \alpha}(0)=1$, and by \cite[(2.5)]{Zla},
$$ h_{g,\sqrt \alpha}(v)\le v\quad \forall\, v\in [0,\infty).
$$

For any $\lambda\in (\lambda_0,\lambda_1)$, let
$$
u^+(t,x)=\phi_\lambda(x) e^{\lambda t}\quad {\rm and}\quad
u^-(t,x)=h_{g,\sqrt \alpha}(\phi_{\lambda}(x) e^{\lambda t}).
$$
By \cite[Lemma 2.1, Theorem 1.1]{Zla}, $u^+(t,x)$ is a super-solution of \eqref{main-eq} and
$u^-(t,x)$ is a sub-solution of \eqref{main-eq}, and there is a transition wave $u(t,x)=U_\lambda(t,x)$ of \eqref{main-eq}
satisfying
\begin{equation}
\label{zl-eq1}
u^-(t,x)\le U_\lambda(t,x)\le u^+(t,x).
\end{equation}
We show that this transition wave is asymptotically stable by applying the arguments in the proof of
Theorem \ref{main-thm-general-case}.

To this end,  first, let
$$
w^+(t,x)=u^+(t,x)-u^-(t,x)(\ge 0).
$$
We have
\begin{align*}
w^+_t(t,x)-w^+_{xx}-a(x)w^+(t,x)&=-(u^-_t-u^-_{xx}-a(x)u^-(t,x))\\
&\ge -u^-(t,x) f(x,u^-(t,x))+a(x) u^-(t,x)\\
&=u^-(t,x)(a(x)-f(x,u^-(t,x))\\
&\ge 0.
\end{align*}
This implies that
$$
d_1 w^+_t(t,x)\ge d_1 w^+_{xx}+d^1 a(x) w^+(t,x)
$$
for any $d_1>0$. Let
$$
\phi(t,x)=u^+(t,x),\quad \phi_1(t,x)=w^+(t,x).
$$
We have that $u(t,x)=d \phi(t,x)+d_1 \phi_1(t,x)$ is a super-solution of \eqref{main-eq} for any $d,d_1>0$.
Note that
$$
\lim_{x\to\infty}\phi(t,x)=\lim_{x\to\infty}\phi_1(t,x)=0\quad {\rm and} \quad \lim_{x\to -\infty} \phi(t,x)=\lim_{x\to -\infty}\phi_1(t,x)=\infty
$$
locally uniformly in $t$.

Next, for any $M>0$, let
$$
g_M(u)=g(u)-Mu^2
$$
and $U_{g_M,\sqrt \alpha}$ be the traveling front profile of
$$
u_t=u_{xx}+g_M(u)
$$
with propagating speed $\sqrt \alpha+\frac{1}{\sqrt\alpha}$ and $\lim_{x\to\infty} U_{g_M,\sqrt \alpha}(x) e^{\sqrt \alpha x}=1$.
Let
$$
h_{g_M,\sqrt \alpha}(v)=U_{g_M,\sqrt a}(-\alpha^{-\frac{1}{2}}\ln v)
$$
for $v>0$ and $h_{g_M,\sqrt \alpha}(0)=0$. Then
$$
h_{g_M,\sqrt\alpha}^{'}(0)=1\quad {\rm and}\quad h_{g_M,\sqrt\alpha}(v)\le h_{g,\sqrt\alpha}(v)\le v.
$$
Similarly, by \cite[Lemma 2.1, Theorem 1.1]{Zla},  we have that $\psi_M(t,x):=h_{g_M,\sqrt \alpha}(\phi_{\lambda}(x) e^{\lambda t})$ is a sub-solution of \eqref{main-eq}. This also implies that $u=d \psi_M(t,x)$ is  a sub-solution of \eqref{main-eq} for any $0<d\le 1$.

Now, note that for any $M>0$ and $d_1>0$,
\begin{equation}
\label{zl-eq2}
\psi_M(t,x)\le U_\lambda(t,x)\le \phi(t,x)+d_1\phi_1(t,x)\quad \forall\,\, t,x\in\RR
\end{equation}
and
$$
\lim_{M\to\infty} \psi_M(t,x)=0
$$
uniformly in $(t,x)\in\RR\times\RR$. Hence
for any given $u_0$ satisfying \eqref{cond-eq3} and \eqref{cond-eq4}, for any $\epsilon>0$, there are $M>0$ and $d_1>0$ such that
\begin{equation*}
(1-\epsilon)\psi_M(t_0,x)\le u_0(x)\le (1+\epsilon)\phi(t_0,x)+d_1\phi_1(t_0,x)
\end{equation*}
and then
\begin{equation}
\label{zl-eq3}
(1-\epsilon)\psi_M(t,x)\le u(t,x;t_0,u_0)\le (1+\epsilon)\phi(t,x)+d_1\phi_1(t,x)\quad \forall\, \, t\ge t_0,\,\,\, x\in\RR.
\end{equation}
By \eqref{zl-eq2} and \eqref{zl-eq3}, we have that
\begin{align}
\label{zl-eq4}
u(t,x;t_0,u_0)\le (1+\epsilon) U_\lambda(t,x)\frac{\phi(t,x)}{\psi_M(t,x)}\Big(1+ d_1\frac{\phi_1(t,x)}{\phi(t,x)}\Big)\quad \forall\,\, t\ge t_0,\,\, x\in\RR
\end{align}
and
\begin{equation}
\label{zl-eq5}
u(t,x;t_0,u_0)\ge (1-\epsilon) U_\lambda(t,x)\frac{\psi_M(t,x)}{\phi(t,x)}\Big(1+ d_1\frac{\phi_1(t,x)}{\phi(t,x)}\Big)^{-1}\quad \forall\,\, t\ge t_0,\,\, x\in\RR.
\end{equation}

Let $X(t)$ be an interface location of $U_\lambda(t,x)$. By \cite[Proposition 4.2]{HaRos}, $X(t)$ satisfies \eqref{cond-eq-2}.
Note that
$$
\lim_{x\to\infty} U_\lambda(t,x+X(t))=0
$$
uniformly in $t\in\RR$. By \eqref{zl-eq1} and \eqref{zl-eq2},
$$
\lim_{x\to\infty} h_{g_M,\sqrt \alpha}(\phi_\lambda(x+X(t))e^{\lambda t})=\lim_{x\to\infty} h_{g,\alpha}(\phi_\lambda(x+X(t))e^{\lambda t})=0
$$
uniformly in $t\in\RR$. This implies that
$$
\lim_{x\to\infty} \phi_\lambda(x+X(t))e^{\lambda t}=0
$$
uniformly in $t\in\RR$. Hence
$$
\lim_{x\to\infty}\frac{\psi_M(t,x+X(t))}{\phi(t,x+X(t))}=\lim_{x\to\infty} \frac{h_{g_M,\sqrt \alpha}(\phi_\lambda(x+X(t))e^{\lambda t})}{\phi_\lambda(x+X(t))e^{\lambda t}}=1
$$
and
$$ \lim_{x\to\infty}\frac{\phi_1(t,x+X(t))}{\phi(t,x+X(t))}=\lim_{x\to\infty} \Big( 1-\frac{h_{g,\sqrt \alpha}(\phi_\lambda(x+X(t))e^{\lambda t})}{\phi_{\lambda}(x+X(t))e^{\lambda t}}\Big)=0
$$
uniformly in $t\in\RR$.
This together with \eqref{zl-eq4} and \eqref{zl-eq5} implies that, for any given $\epsilon_0>0$ with $\frac{\epsilon_0}{4+2\epsilon_0}>\epsilon$,
for any $s\ge 0$, there is $x_s(\ge X(t_0+s))$
 such that
 \begin{equation}
 \label{zl-eq6}
 \sup_{s\ge 0, t\in [t_0+s,t_0+s+\tau]}|x_s-X(t)|<\infty
 \end{equation}
 and
\begin{align}
\label{zl-eq7}
\frac{1}{1+\epsilon_0/2}U_\lambda(t,x)\le u(t,x;t_0,u_0)\le (1+\epsilon_0/2)U_\lambda(t,x)\quad \forall\,\, t\in [t_0+s,t_0+s+\tau],\,\, x\ge x_s.
\end{align}
It then follows from the arguments after \eqref{eqq-1} in the proof of  Theorem \ref{main-thm-general-case} that,
for any $\epsilon>0$,
\begin{equation*}
\rho(u(t+t_0,\cdot;t_0,u_0),U_\lambda(t+t_0,\cdot))<\epsilon\quad \text{for
some}\,\, t>0.
\end{equation*}
 Then by Proposition \ref{part-metric-prop}(1),
 \begin{equation*}
\rho(u(t+t_0,\cdot;t_0,u_0),U_\lambda(t+t_0,\cdot))<\epsilon\quad \text{for
some}\,\, t\gg 1.
\end{equation*}
Therefore,
$u=U_\lambda(t,x)$ is asymptotically stable.
\end{proof}

\begin{proof}[Proof of Theorem \ref{main-thm-special-case} (4)]
Note that the existence of transition waves is established in \cite{RaShZh}.
In the following, we outline the construction of transition waves from \cite{RaShZh} and  show that they satisfy the conditions in Theorem \ref{main-thm-general-case} and hence are asymptotically stable.

First of all, consider
\begin{equation}
\label{linear-eq2}
v_t=\int_{\RR} e^{-\mu(y-x)}\kappa(y-x)v(t,y)dy-v(t,x)+a(t,x)v(t,x),\quad x\in\RR,
\end{equation}
where $\mu\in\RR$, $a(t,x)=f(t,x,0)$.
 Assume that $a(t,x+p)=a(t+T,x)=a(t,x)$.
By \cite[Propositions 3.2, 3.4, 3.5]{RaShZh}, we have
\begin{itemize}
\item[(i)] For given $\mu>0$, there are $\lambda(\mu)\in\RR$ and  a continuous function $v(t,x;\mu)$ such that
$$v(t+T,x;\mu)=v(t,x+p;\mu)=v(t,x;\mu),\quad \inf_{(t,x)\in\RR\times\RR} v(t,x;\mu)>0,
$$
 $$\|v(\cdot,\cdot;\mu)\|_\infty=1,
 $$
 and
$$
v(t,x):=e^{\lambda(\mu) t} v(t,\cdot;\mu)
$$
is a solution of \eqref{linear-eq2}.

\item[(ii)] There is $\mu^*>0$ such that
$$
\frac{\lambda(\mu^*)}{\mu^*}=\inf_{\mu>0}\frac{\lambda(\mu)}{\mu}
$$
and
$$
\frac{\lambda(\mu)}{\mu}>\frac{\lambda(\mu^*)}{\mu^*}\quad {\rm for}\quad 0<\mu<\mu^*.
$$
\end{itemize}

\noindent Let $c^*=\frac{\lambda(\mu^*)}{\mu^*}$. We show that  Theorem \ref{main-thm-special-case} (4) holds with this $c^*$.

To this end, for given $0<\mu<\mu^*$, choose $\mu_1$ such that $\mu<\mu_1<\min\{\mu^*,2\mu\}$. Let $c_\mu=\frac{\lambda(\mu)}{\mu}$.
Let
$$
\phi(t,x)=e^{-\mu (x-\frac{1}{\mu}c_\mu t)}v(t,x;\mu),\quad \phi_1(t,x)=e^{-\mu_1 (x-\frac{1}{\mu}c_\mu t)}v(t,x;\mu_1),
$$
where $v(t,x;\mu)$ and $v(t,x;\mu_1)$ are as in (i).
By the arguments of \cite[Propositions 3.2, 3.5]{ShZh1} and \cite[Propositions 5.1,5.2]{RaShZh},
 there is $d_0>0$ such that for any $0<d<1$, $d_1>d_0 d$, and any $t_0\in\RR$, $u_0\in C_{\rm unif}^b(\RR)$ ($u_0(x)\ge 0$)  satisfying
\begin{align*}
d\phi(t_0,x)-d_1\phi_1(t_0,x)\le u_0(x)\le d\phi(t_0,x)+d_1 \phi_1(t_0,x),
\end{align*}
there holds
\begin{align}
\label{IVP-eq3}
d\phi(t,x)-d_1 \phi_1(t,x)
\le u(t,x;t_0,u_0)\le d\phi(t,x)+d_1 \phi_1(t,x)\quad \forall\,\, t\ge t_0.
\end{align}

Fix $0<d^*<1$ and $d_1^*>d_0 d^*$. By the arguments of \cite[Theorem 2.4]{ShZh1} and \cite[Theorem 5.1]{RaShZh},  there is a
uniformly continuous  periodic
transition wave solution $u=U(t,x)$ satisfying
\begin{align}
\label{transition-wave-eq2}
d^*\phi(t,x)-d_1^*\phi_1(t,x)\le U(t,x)
\le d^*\phi(t,x)+d_1^* \phi_1(t,x).
\end{align}
Clearly, $X(t)=\frac{c_\mu}{\mu} t$ is an interface location of $U(t,x)$ and satisfies \eqref{cond-eq-2},  and $\phi(t,x)$, $\phi_1(t,x)$ satisfy \eqref{cond-eq-1} and \eqref{cond-eq0}.
By \eqref{IVP-eq3}, \eqref{transition-wave-eq2}, and Theorem \ref{main-thm-general-case}, we have that $u=U(t,x)$ is a periodic wave solution with  speed $c_\mu=\frac{\lambda(\mu)}{\mu}$ and is asymptotically stable.
\end{proof}

\end{document}